\documentclass[11pt,leqno]{amsart}

\usepackage{hyperref}
\usepackage{a4,amssymb,amsthm,amscd,amsmath,verbatim,url,enumerate,cleveref}
\usepackage{color}

\usepackage{enumitem}

\usepackage{todonotes}

\usepackage[small,nohug,heads=vee]{diagrams}
\diagramstyle[labelstyle=\scriptstyle]

\makeatletter
\def\foo#1\endgraf\unskip#2\foo{\def\row@to@buffer{#1\endgraf\unskip\unskip#2}}
\expandafter\foo\row@to@buffer\foo
\makeatother

\newcommand{\N}{\mathbb{N}}
\newcommand{\Z}{\mathbb Z}
\newcommand{\C}{\mathbb C}
\newcommand{\Sim}{\mathsf{Sim}}

\renewcommand{\leq}{\leqslant}
\renewcommand{\geq}{\geqslant}
\newcommand{\diman}{\mathsf{diman}}
\renewcommand{\dim}{\mathsf{dim}}
\newcommand{\W}{\mathsf{W}}
\newcommand{\I}{\mathsf{I}}
\renewcommand{\P}{\mathsf{P}}
\newcommand{\GP}{\mathsf{GP}}
\newcommand{\GA}{\mathsf{GA}}

\newcommand{\an}{\mathsf{an}}
\newcommand{\R}{\mathbb{R}}

\renewcommand{\det}{\operatorname{\mathsf{det}}}
\renewcommand{\mod}{\operatorname{\,\mathsf{mod}}\,}

\numberwithin{equation}{section}
\newtheorem{thm}[equation]{Theorem}
\newtheorem*{thm*}{Theorem}
\newtheorem{prop}[equation]{Proposition}
\newtheorem*{prop*}{Proposition}
\newtheorem{cor}[equation]{Corollary}
\newtheorem*{cor*}{Corollary}
\newtheorem{conj}[equation]{Conjecture}
\newtheorem*{conj*}{Conjecture}
\newtheorem{lem}[equation]{Lemma}
\newtheorem*{lem*}{Lemma}

\theoremstyle{definition}
\newtheorem{de}[equation]{Definition}
\newtheorem*{de*}{Definition}

\newtheorem{rem}[equation]{Remark}
\newtheorem*{rem*}{Remark}
\renewenvironment{proof}{\par\noindent {\em Proof: }}{\hfill$\Box$\medskip}
\theoremstyle{plain}

\newcommand{\Pfister}[1]{\langle\!\langle #1\rangle\!\rangle}
\newcommand{\laurent}[1]{(\!(#1)\!)}
\newcommand{\qf}[1]{\langle #1\rangle}
\usepackage{graphicx}

\makeatletter
\newcommand{\bigperp}{%
  \mathop{\mathpalette\bigp@rp\relax}%
  \displaylimits
}

\newcommand{\bigp@rp}[2]{%
  \vcenter{
    \m@th\hbox{\scalebox{\ifx#1\displaystyle2.1\else1.5\fi}{$#1\perp$}}
  }%
}
\makeatother

\title{On Generalised Albert Forms over Discretely Valued Fields}
\date{\today}
\author{Nico Lorenz}
\address{Fakult\"at f\"ur Mathematik, Ruhr-Universit\"at Bochum, Universit\"atsstra\ss e 150, 44801 Bochum, Deutschland}
\email{nico.lorenz@ruhr-uni-bochum.de}
\begin{document}

\begin{abstract}\noindent
    For a discrete valuation ring $R$ with quotient field $K$ and residue field $F$ both of characteristic not 2, we study low-dimensional quadratic forms with Witt class in the $n$-th power of the fundamental ideal of $F$ resp. $K$ and point out connections between forms over these fields.
    We analyse the minimal number of Pfister forms such that a given form is Witt equivalent to the sum of these and study forms congruent modulo a higher power of the fundamental ideal towards similarity.
\noindent

\medskip\noindent
{\sc{Classification (MSC 2020): 11E81, 12J10} }


\end{abstract}
\maketitle

\section{Introduction}

Let $F$ be a field of characteristic not 2. 
The investigation of low dimensional quadratic forms $\varphi$ whose Witt class $[\varphi]$ lies in $\I^n(F)$, the $n$-th power of the fundamental ideal $\I(F)$ is a longstanding problem in the algebraic theory of quadratic forms.

The first breakthrough was attended by J. K. Arason and A. Pfister with the proof of their \emph{Hauptsatz}, which gives a lower bound on the dimension of such forms.

\begin{thm}[Arason-Pfister Hauptsatz, {\cite[Hauptsatz]{ArasonPfister71}}]\label{thm:APH}
    Let $\varphi$ be an anisotropic quadratic form of positive dimension with $[\varphi]\in\I^n(F)$.
    Then $\dim(\varphi)\geq 2^n$ and we have equality if and only if $\varphi$ is similar to a Pfister form.
\end{thm}

In order to find other possible dimensions of anisotropic forms with Witt class in $\I^n(F)$, the concept of \emph{linkage}, first introduced in \cite{MR0283004}, is very fruitful.
It determines the dimensions of forms that are Witt equivalent to a sum of two $n$-fold Pfister forms.
For later use, we cite a more general version that includes Pfister forms of different foldness.

\begin{thm}{\cite[Lemma 3.2]{HoffmannTwistedPfister}}\label{thm:LinkTwoPfisterForms}
	Let $\sigma\in \P_n(F)$ and $\pi\in \P_m(F)$ be anisotropic Pfister forms for some $m,n\in\N$ with $1\leq m\leq n$ and $a,b\in F^\ast$. 
    For the Witt index $i_W$, we then have
	\[i:=i_W(a\sigma\perp b\pi)\in\{0\}\cup\{2^r\mid 0\leq r\leq m\}.\]
	Further, we have $i\geq1$ if and only if there is some $x\in F^\ast$ with 
	\[(a\sigma\perp b\pi)_\text{an}\cong x(\sigma\perp-\pi)_\text{an}.\]
	If $i=2^r\geq1$ then there exist $\alpha\in \P_r(F), \sigma_1\in \P_{n-r}(F)$ and $\pi_1\in \P_{m-r}(F)$ such that we have
	\[\sigma\cong\alpha\otimes\sigma_1\text{   and   }\pi\cong\alpha\otimes\pi_1.\]
	In this case, $r$ is called the \emph{linkage number} of $\sigma$ and $\pi$ and $\alpha$ is called a \emph{link} of $\sigma$ and $\pi$. 
    If we have $n=m$ and $r\geq n-1$, we say $\sigma$ and $\pi$ are \emph{linked}.
\end{thm}

\Cref{thm:LinkTwoPfisterForms} shows that anisotropic forms with Witt class in $\I^n(F)$ that are Witt equvialent to a sum of two $n$-fold Pfister forms have dimension $2^{n+1}$ or $2^{n+1}-2^k$ for some $k\in\{1,\ldots, n+1\}$.
It is natural to ask whether these are all dimensions that can be realized for $\I^n$-forms of dimension $\leq 2^{n+1}$.
In particular, in addition to the first gap $(0, 2^n)$ provided by the Arason-Pfister Hauptsatz \ref{thm:APH}, the second gap $(2^n, 2^n+2^{n-1})$ was subject to research for a long time. 
While the cases for $n\leq2$ are clear, Pfister showed that this is a true gap for $n=3$ already in \cite[Satz 14]{Pfister}.
This was extended to $n=4$ by D. Hoffmann in \cite[Main Theorem]{HoffmannI4}, then for all $n$ for fields of characteristic 0 by A. Vishik in \cite{MR1788358}.
It could be verified for all characteristics $\neq 2$ by N. Karpenko in \cite{MR2058676} that there are no anisotropic quadratic forms of dimension between $2^n$ and $2^{n}+2^{n-1}$ whose Witt class lies in $\I^n(F)$.

Finally, the complete classification of possible dimensions of forms in $\I^n(F)$ was given again by N. Karpenko, showing that the ones prescibed by \Cref{thm:LinkTwoPfisterForms} are exactly the ones that occur:

\begin{thm}[Gap Theorem {\cite{KarpenkoHoles}}]\label{thm:Gaps}
    Let $\varphi$ an anisotropic quadratic form with $\dim(\varphi)<2^{n+1}$ and $[\varphi]\in\I^n(F)$.
    Then there is some $k\in\{1,\ldots, n+1\}$ with $\dim(\varphi)=2^{n+1}-2^k$.
\end{thm}

After having classified all the possible dimensions, it is natural to ask for further structural descriptions of such forms.
Again, this was already done by A. Pfister for $n=3$ resp. D. Hoffmann for $n=4$.

\begin{thm}[{\cite[Satz 14, Zusatz]{Pfister}}]
	Let $\varphi$ be a quadratic form with $\dim(\varphi)=12$ and $[\varphi]\in \I^3(F)$. 
    Then there are $x\in F^\ast$ and an Albert form $\alpha$ over $F$ with $\varphi\cong\Pfister x\otimes\alpha$.
    In particular, $\varphi$ is Witt equivalent to a scalar multiple of a sum of two Pfister forms.
\end{thm}

\begin{thm}[{\cite{HoffmannI4}}]
    Let $\varphi$ be an anisotropic form of dimension $\dim(\varphi)=24$ with $[\varphi]\in\I^4(F)$. 
    Then, there exists a Pfister form $\pi\in\GP_2(F)$, an Albert form $\alpha$ and $\rho\in\GP_4(F)$ such that for the Witt classes, we have
    \[[\varphi]=[\pi\otimes\alpha]+[\rho].\]
\end{thm}

With \Cref{thm:LinkTwoPfisterForms} in mind, another approach to studying low dimensional forms with Witt class in $\I^n(F)$ is to investigate sums of two Pfister forms. 
We introduce the following terminology.

\begin{de}
    Let $\varphi$ be an anisotropic quadratic form with $[\varphi]\in\I^n(F)$ and $2^n<\dim(\varphi)<2^{n+1}$.
    The form $\varphi$ is said to be a \emph{generalised Albert form (of degree $n$)}, if there are $\pi_1,\pi_2\in\GP_n(F)$ such that 
    \[[\varphi]=[\pi_1]+[\pi_2].\]
    We further denote the set of generalised Albert forms of degree $n$ over $F$ by $\GA_n(F)$.
\end{de}

Such forms also have been investigated by several authors and we now briefly collect some highlights:

In \cite{MR1603857}, generalised Albert forms were investigated with regard towards their splitting properties. 

In \cite[Theorem 4.9]{MR1668530}, O. Izhboldin constructed for each $n\geq3$ and each $k\in\{1,\ldots, n-2\}$ anisotropic generalised Albert forms $\alpha, \beta$ both of dimension $2^{n+1}-2^{k}$ that are not similar, but such that $[\alpha]\equiv[\beta]\mod\I^{n+1}(F)$.

The case $k=n-1$ was not answered in Izhboldin's paper, but D. Hoffmann was able to handle such forms.
Using the Gap \Cref{thm:Gaps} and \cite[Proposition 1]{HoffmannSimilarity}, it follows that Izhboldin's result cannot be generalised to the case $k=n-1$.
Stated more explicitly, the combination of the two results yields the following:

\begin{thm}[D. Hoffmann]\label{thm:GA->Sim}
    Let $\alpha_1,\alpha_2$ be generalised Albert forms over $F$ of dimension $2^n+2^{n-1}$ for some $n\geq2$ such that 
    \[[\alpha_1]\equiv[\alpha_2]\mod \I^{n+1}(F).\]
    Then $\alpha_1$ and $\alpha_2$ are similar.
\end{thm}

We will now have a closer look at all the forms of dimension $2^n+2^{n-1}$ whose Witt classes lie in $\I^n(F)$.
In \cite[Conjecture 2]{HoffmannI4}, it is explicitly conjectured that for $n=4$, these forms are generalised Albert forms. 
As no counterexamples are known so far, this conjecture can clearly be broadened to cover all positive integers $n$:

\begin{conj}\label{conj:GenAl}
    Let $n\geq2$ and let $\varphi$ be a quadratic form with $\dim(\varphi)=2^n+2^{n-1}$ and $[\varphi]\in\I^n(F)$.
    Then $\varphi$ is a generalised Albert form.
\end{conj}

Note that this conjecture cannot be transferred to generalised Albert forms of higher dimension in some fixed power of the fundamental ideal. 
For example, there are fields $F$ and 14-dimensional forms $\varphi$ with Witt class in $\I^3(F)$ that are not generalised Albert forms as found independently in \cite{HoffmannTignol} and \cite{IzhboldinKarpenko}.
See also the upcoming \Cref{rem:NoHigherDimension} for some further information.

The aim of this note is now to study the validity of \Cref{conj:GenAl} with respect to discretely valued fields. 
Therefore we consider a discrete valuation ring $R$ with residue field $F$ and quotient field $K$ both of characteristic not 2.
We will study forms of dimension $2^n+2^{n-1}$ in $\I^n(F)$ resp. $\I^n(K)$ and prove going-up and going-down theorems that describe how structural properties between such forms over $F$ resp. $K$ can be transferred to the other field.
As a main result, we will show in \Cref{thm:GAUp}
that if \Cref{conj:GenAl} holds for $F$ for $n-1$ and $n$ then \Cref{conj:GenAl} holds for $K$ for $n$.
In the light of \Cref{thm:GA->Sim} we will further consider the question whether two forms of dimension $2^n+2^{n-1}$ whose Witt class lies in $\I^n$ and are congruent modulo $\I^{n+1}$ are similar and also prove relations between $F$ and $K$. 
 
The current work can be seen as an accompanying work of a collaboration with K. Becher \cite{BecherLorenz}.
The yet unpublished work focuses on generalised Albert forms in \emph{abstract Witt rings} as considered by M . Marshall in \cite{MR0674651}. 
Thus no transcendental or algebro-geometric methods are at hand which were used to deduce the above results. 
Thus, the emphasis of the other article lies more on the development of methods that mimic these techniques in the setting of abstract Witt rings.

\section{Preliminaries}

Throughout this article, all fields are assumed to be of characteristic not 2. 
By a \emph{quadratic form} or \emph{form} for short, we will always mean a nondegenerate quadratic form of finite dimension.
The \emph{Witt ring} of a field $F$ is denoted by $\W(F)$ and consists of the \emph{Witt classes} of quadratic forms.
A system of representatives of the Witt classes is given by the isomorphism classes of anisotropic forms.
The Witt class of a form $\varphi$ is denoted by $[\varphi]$.
Isometry of forms is denoted by $\cong$ while equality of Witt classes is denoted by $=$.
We further use $\perp,\otimes$ to denote orthogonal sum and tensor product of quadratic forms and use $+,\cdot$ for the induced operations in the Witt ring.
The dimension of the anisotropic part $\varphi_\an$ of a quadratic form $\varphi$ is denoted by 
\[\diman(\varphi)=\dim(\varphi_\an).\]

A special emphasis will lie on the \emph{fundamental ideal} $\I(F)$, which is generated by the Witt classes of \emph{$1$-fold Pfister forms}, i.e. forms of the shape $\Pfister{a}=\qf{1,-a}$ for some $a\in F^\ast$.
The $n$-th power of the fundamental ideal is denoted by $\I^n(F)$ and is clearly generated by the classes of \emph{$n$-fold Pfister forms} $\Pfister{a_1,\ldots, a_n}=\Pfister{a_1}\otimes\ldots\otimes\Pfister{a_n}$.

The set of all forms that are isometric to an $n$-fold Pfister form is denoted by $\P_n(F)$ and the set of all forms that are similar to an $n$-fold Pfister form is denoted by $\GP_n(F)$.
Slightly abusing terminology, we call forms in $\GP_n(F)$ also Pfister forms.

For a form $\varphi$ with $[\varphi]\in\I^n(F)$, an invariant that is important to us is the \emph{$n$-Pfister number} $\GP_n(\varphi)$ defined as
\[\GP_n(\varphi)=\min\{k\in\N\mid \exists\pi_1,\ldots,\pi_k\in\GP_n(F):[\varphi]=[\pi_1]+\ldots+[\pi_k]\}.\]
The Pfister number of quadratic forms has been investigated by many researchers, see \cite{HoffmannTignol}, \cite{IzhboldinKarpenko}, \cite{ParimalaSureshTignol},
\cite{MR2630047}, \cite{Raczek2013}, \cite{Karpenko2017}, \cite{MR4554566} among others. 

As an example for somehow easy fields in which we can determine the Pfister numbers precisely, we have \emph{$n$-linked} fields, which are a straightforward generalization of the usual linked fields (which are precisely the 2-linked fields).

\begin{thm}[\cite{HoffmannSimpleDecomposition}]\label{SimpleDec}
	For any field $F$ and any $n\geq 2$, the following are equivalent:
	\begin{enumerate}[label=(\roman*)]
		\item for any anisotropic form $\varphi$ with $[\varphi]\in \I^n(F)$ there are $\pi\in \P_{n-1}(F)$ and $\tau\in \I(F)$ with $\varphi\cong \pi\otimes\tau$;
		\item\label{SimpleDec2} any anisotropic form $\varphi$ with $[\varphi]\in \I^n(F)$ is isometric to a sum of forms in $\GP_n(F)$;
		\item $F$ is $n$-linked, i.e. for any $\pi_1,\pi_2\in \P_n(F)$, there is some $\pi\in \P_{n-1}(F)$ and $a,b\in F^\ast$ with $\varphi\cong\Pfister a\otimes \sigma$ and $\psi\cong\Pfister b\otimes \sigma$.
	\end{enumerate}
\end{thm}

The second statement in the above result immediately implies the following:

\begin{cor}\label{cor:nlinkedPfisterNumber}
	Let $F$ be an $n$-linked field and let $\varphi$ be an anisotropic quadratic form of dimension $\dim (\varphi)=d$ with $[\varphi]\in \I^n(F)$. 
    Then $d$ is divisible by $2^n$ and we have $\GP_n(\varphi)=\frac d{2^n}$.
\end{cor}

There are some easy bounds for the Pfister number if we have divisibility by some Pfister form.
We start with the case of 1-fold Pfister forms.

\begin{prop}\label{prop:PfisterNumberMult}
	Let $[\psi]\in \I^n(F)$ and $a\in F^\ast$. 
    For $\varphi=\Pfister a\otimes\psi$, we have $[\varphi]\in\I^{n+1}(F)$, and
	\[{\GP}_{n+1}(\varphi)\leq{\GP}_n(\psi).\]
\end{prop}
\begin{proof}
	It is clear that we have $[\varphi]\in \I^{n+1}(F)$ as we have $[\psi]\in \I^n(F)$ and $[\Pfister a]\in \I(F)$. 
    Let now $\pi_1,\ldots,\pi_k\in \GP_n(F)$ with $[\psi]=[\pi_1]+\ldots+[\pi_k]$. 
    We then have
	\[[\varphi]=[\Pfister a\otimes\pi_1]+\ldots+[\Pfister a\otimes\pi_k]\]
	with $\Pfister a\otimes\pi_\ell\in \GP_{n+1}(F)$ for all $\ell\in\{1,\ldots, k\}$ as desired.
\end{proof}

If the foldness of the Pfister factor is high in relation to the power of the fundamental ideal in which the given quadratic form lives, the following observation will be useful in the sequel.

\begin{lem}\label{lem:WLOGI2}
	Let $n\in\N$ be an integer with $n\geq 3$ and $\varphi$ be an anisotropic form of positive dimension that is divisible by some $\pi\in\GP_{n-2}(F)$  and such that $[\varphi]\in \I^n(F)$. 
    Then, we have $2\mid \frac{\dim(\varphi)}{\dim(\pi)}$ and there is some quadratic form $\sigma$ with $[\sigma]\in \I^2(F)$ and $\varphi\cong\pi\otimes\sigma$.
    
	In particular, we have
	\[\GP_n(\varphi)\leq \GP_2(\sigma)\leq \frac{\dim(\sigma)-2}2.\]
\end{lem}
\begin{proof}
	By assumption, we can find a quadratic form $\tau$ of dimension $d$ over $F$ with $\varphi\cong\pi\otimes\tau$. 
	We first assume further that $2\mid \dim(\tau) = \frac{\dim(\varphi)}{\dim(\pi)}$ and will show at the and of the proof that this is necessary. 
    If we already have $[\tau]\in \I^2(F)$, we are done.
    Otherwise, we write $\tau\cong\qf x\perp\tau'$ for some $x\in F^\ast$ and a quadratic form $\tau'$.
    
    For $a:=(-1)^{\frac d2}\cdot \det(\tau')$, we put $\sigma:=\qf{a}\perp\tau'$. 
    Then $\sigma$ is also of dimension $d$ and have $\det(\sigma)\in(-1)^{\frac d2}F^\ast$ and therefore $[\sigma]\in \I^2(F)$.
	In $\W(F)$ we thus have
	\begin{align*}
		[\varphi]&=[\psi\otimes\tau]=[\pi\otimes(\qf x\perp\tau')]\\
		&=[\pi\otimes(\tau'\perp\qf{a}\perp\qf{x,-a})]=[\pi\otimes\sigma]+[\pi\otimes\qf{x, -a}].
	\end{align*}
	As we have $[\varphi], [\pi\otimes\sigma]\in \I^n(F)$, we also have $[\pi\otimes\qf{x, -a}]\in \I^n(F)$. 
    But the {Arason-Pfister Hauptsatz} \ref{thm:APH} now implies the latter form to be hyperbolic. 
    We thus have $[\varphi]=[\pi\otimes\sigma]$ in $\W(F)$ and as the dimensions of both forms coincide, we even have the isometry $\varphi\cong\pi\otimes\sigma$.

    For $\dim(\tau)$ odd, we put $a=(-1)^{\frac d2+1}\cdot\det(\tau)$. 
    We then have
	\begin{align*}
		[\varphi]=[\pi\otimes \tau]=[\pi\otimes(\tau\perp\qf{a,-a})]=[\pi\otimes(\tau\perp\qf{a})]-[a\pi].
	\end{align*}
	As above, for $\sigma:=\tau\perp\qf{a}$, we have $[\sigma]\in \I^2(F)$ and thus ${[\pi\otimes\sigma]\in \I^n(F)}$. 
    Invoking the Arason-Pfister Hauptsatz \ref{thm:APH} again, this implies $-a\pi$ to be hyperbolic. 
    Hence $\pi$ would be hyperbolic, contradicting the hypothesis.
    
	As now the first claim is established, the second one is clear using the well known bound on the 2-Pfister number, see \cite[Chapter X. Exercise 4]{Lam2005}.
\end{proof}

As already stated in the introduction, generalised Albert forms of dimension $2^n+2^{n-1}$ will be those forms of main interest.
We thus devote the rest of this section to the study of such forms and start by proving a generalization of \cite[Proposition 4.1]{HoffmannI4} for arbitrary $n\geq2$. 
We will mainly use the same techniques as in the original article, but we will further use the Gap \Cref{thm:Gaps} that was only a conjecture and not known to be true when \cite{HoffmannI4} was published.

Before stating and proving the result, we would like to note that the case $n=2$ is trivial and that the case $n=3$ can essentially be found in \cite[Satz 14, Zusatz]{Pfister}. 

\begin{prop}\label{prop:GAclassification}
	Let $[\varphi]\in \I^n(F)$ with $\varphi$ an anisotropic form of dimension $\dim(\varphi)=2^n+2^{n-1}$ for some $n\in\N$ with $n\geq 2$.
    Then the following are equivalent:
	\begin{enumerate}[label=(\roman*)]
		\item\label{prop:GAclassification1} there are $\pi_1,\pi_2\in \GP_n(F)$ with $[\varphi]=[\pi_1]+[\pi_2]$, i.e. $\GP_n(\varphi)=2$;
		\item\label{prop:GAclassification2} there is some $\pi\in \GP_{n-2}(F)$ and an Albert form $\alpha$ with $\varphi\cong \pi\otimes\alpha$;
		\item\label{prop:GAclassification3} there are $\sigma_1,\sigma_2,\sigma_3\in \GP_{n-1}(F)$ with $\varphi\cong \sigma_1\perp\sigma_2\perp\sigma_3$, i.e. $\GP_{n-1}(\varphi)=3$;
		\item\label{prop:GAclassification4} there is some $\sigma\in \GP_{n-1}(F)$ with $\sigma\subseteq \varphi$;
		\item there is some Pfister neighbor $\psi\subseteq\varphi$ of dimension $2^{n-1}+1$.
	\end{enumerate}
\end{prop}
\begin{proof}
	\begin{description}[font=\normalfont]
		\item[(i)$\Rightarrow$(ii)] According to \Cref{thm:LinkTwoPfisterForms} the $n$-fold Pfister forms that $\pi_1$ respectively $\pi_2$ are similar to have linkage number $n-2$, i.e. there is some $\pi\in \GP_{n-2}(F)$ that divides both $\pi_1$ and $\pi_2$ and thus $\varphi$. 
        By counting dimensions, we have $\varphi\cong\pi\otimes\sigma$ for some quadratic form $\alpha$ of dimension $6$. 
        By \Cref{lem:WLOGI2} we may assume $[\alpha]\in\I^2(F)$, i.e. $\alpha$ is an Albert form.
		\item[(ii)$\Rightarrow$(iii)] We decompose $\alpha=\alpha_1\perp\alpha_2\perp\alpha_3$ with binary forms $\alpha_1,\alpha_2,\alpha_3$. 
        We now just have to put $\sigma_i:=\pi\otimes\alpha_i$ for $i\in\{1,2,3\}$.
		\item[(iii)$\Rightarrow$(iv)] This is trivial.
		\item[(iv)$\Rightarrow$(v)] We write $\varphi=\sigma\perp\qf{x,\ldots}$ for some $x\in F^\ast$. 
        Then $\psi:=\sigma\perp\qf{x}$ is a Pfister neighbor of dimension $2^{n-1}+1$ of the Pfister form $\sigma\perp x\sigma$ and clearly a subform of $\varphi$.
		\item[(v)$\Rightarrow$(i)] As $\psi$ is a Pfister neighbor of dimension $2^{n-1}+1$, there is some $\pi_1\in \GP_n(F)$ with $\psi\subseteq \pi_1$. 
        For $\pi_2:=(\varphi\perp-\pi_1)_\text{an}$, we clearly have 
        \[[\pi_2]=[\varphi]-[\pi_1]\in \I^n(F)\] 
        and this form is of dimension at most 
        \[2^n+2^{n-1}+2^n-2\cdot(2^{n-1}+1)=2^n+2^{n-1}-2.\] 
        By the Gap \Cref{thm:Gaps} and the Arason-Pfister Hauptsatz \ref{thm:APH} we have $\pi_2\in \GP_n(F)$ and the conclusion follows.
	\end{description}
\end{proof}

\begin{rem}
	The equivalence of \ref{prop:GAclassification2} and \ref{prop:GAclassification3} can also readily be deduced from \cite[Chapter X, Linkage Theorem 6.22]{Lam2005}.
\end{rem}

Let $d\in\N$ be a positive integer.
We define
\[\GA_n(F,d)=\{[\varphi]\in\W(F)\mid \diman(\varphi)=d, \varphi\in\GA_n(F)\}\]
and
\[\I^n(F,d)=\{[\varphi]\in\I^n(F)\mid \diman(\varphi)=d\}.\]

Note that while $\GA_n(F)$ is a set of quadratic forms, the sets $\GA_n(F, d)$ are defined to be sets of Witt classes.

For all $n, d$ we clearly have the inclusion $\GA_n(F,d)\subseteq\I^n(F,d)$.
We investigate under which circumstances we have equality and study particularly the case $d=2^n+2^{n-1}$.

For this choice of $d$, we already know by the results stated in the introduction that for all fields $F$, we have $\GA_n(F,d)=\I^n(F,d)$ for $n\in\{2,3\}$.

We will now introduce another property concerning quadratic forms in $\I^n(F)$ of dimension at most $d=2^n+2^{n-1}$ that correlates with the equality $\GA_n(F, d)=\I^n(F,d)$.

\begin{de}
	Let $n\in\N$ be an integer with $n\geq 2$. 
    We say a field $F$ has property $\Sim(n)$ if any two non-zero anisotropic forms $\varphi,\psi$ of dimension at most $2^n+2^{n-1}$ with 
    \[[\varphi],[\psi]\in\I^n(F)\text{ and }[\varphi]\equiv[\psi]\mod\I^{n+1}(F)\] 
are similar.
\end{de}

\begin{rem}
    We clearly could have defined $\Sim(n)$ also for $n=1$.
    This property is always fulfilled trivially since there are no anisotropic forms of dimension 3 in $\I(F)$.
\end{rem}

\begin{rem}
    It is well known that for $\pi_1,\pi_2\in\GP_n(F)$, we have that $\pi_1$ and $\pi_2$ are similar if and only if $[\pi_1]\equiv[\pi_2]\mod\I^{n+1}(F)$, see \cite[Chapter X. Corollary 5.4]{Lam2005}.

    Let now $\varphi$ be an anisotropic form of dimension $2^n+2^{n-1}$ with $[\varphi]\in\I^n(F)$ and $\pi\in\GP_n(F)$ be such that $[\varphi]\equiv[\pi]\mod\I^{n+1}(F)$.
    We then have $[\varphi\perp-\pi]\in\I^{n+1}(F)$ and the Gap \Cref{thm:Gaps} implies 
    \[[\varphi\perp-\pi]=[\pi']\]
    for some $(n+1)$-fold Pfister form $\pi'\in\GP_{n+1}(F)$.
    This implies $[\varphi]=[\pi'\perp\pi]$.
    Considering now the dimensions of the anisotropic parts, we obtain 
    \[i_W(\pi'\perp\pi)=\frac{2^{n+1}+2^n-(2^n+2^{n-1})}2=3\cdot2^{n-2}\] 
    which contradicts the linkage theory, see \Cref{thm:LinkTwoPfisterForms}.

    To verify the property $\Sim(n)$, it is thus enough to consider anisotropic forms of dimension exactly $2^n+2^{n-1}$.
\end{rem}

Using \Cref{thm:GA->Sim}, the following is clear:

\begin{cor}
    Let $n\geq2$ be an integer and $d=2^n+2^{n-1}$.
    If we have $\GA_n(F,d)=\I^n(F,d)$, then $F$ fulfills $\Sim(n)$.
\end{cor}

Note that the validity of $\Sim(2)$, i.e. the fact that (classical) Albert forms whose Witt classes are congruent modulo $\I^3(F)$ are similar, is a very classical result by N. Jacobson, see \cite[Theorem 3.12]{MR0700981}. 
While Jacobson uses the theory of central simple algebras, a proof within the framework of quadratic forms is given in \cite{MR0931742}.

There is a further class of quadratic forms Witt equivalent to the sum of two Pfister forms that will be used later in the proof of \Cref{thm:SimGoingUp}.
This concept was introduced by D. Hoffmann in \cite{HoffmannTwistedPfister}.

\begin{de}
    Let $1\leq m<n$. A form $\varphi$ is called a \emph{twisted $(n,m)$-Pfister form} if there are anisotropic Pfister forms $\sigma\in\P_n(F), \pi\in\P_m(F)$ with linkage number $m-1$ and an element $a\in F^\ast$ such that $[a\varphi]=[\sigma]-[\pi]$.
\end{de}

Note that our definition of twisted Pfister forms is slightly more general than the origin one in that we also include scalar multiples.

Clearly, twisted $(n,m)$-Pfister forms have dimension $2^n$.
Our interest in twisted $(n,m)$-Pfister forms lies in the fact that their similarity can be read off an equivalence modulo $\I^{n+1}$, just as we have already seen for generalised Albert forms in \Cref{thm:GA->Sim}.

\begin{prop}[{\cite[Proposition 2.8]{HoffmannHalfNeighbor}}]\label{prop:TwistedPfisterHalfNeighborSimilar}
    Let $1\leq m<n$ and $\varphi,\psi$ be twisted $(n,m)$-Pfister forms over some field $F$.
    If $[\varphi]\equiv[\psi]\mod\I^{n+1}(F)$ then $\varphi$ and $\psi$ are similar.
\end{prop}

\section{Quadratic Forms over Fields with a Discrete Valuation}

We will now recall some well known facts about quadratic forms over fields carrying a discrete valuation crucial to our further investigations. 
A more detailed exposition on the topic can be found e.g. in \cite[Section 19]{ElmanKarpenkoMerkurjev2008} or \cite[Chapter 6, § 2]{Scharlau}.

Let $R$ be a discrete valuation ring with residue field $F$ and quotient field $K$, both of characteristic not 2.
Let further $t$ be a \emph{uniformizing element} or shortly \emph{uniformizer} of $R$, i.e. a generator of the unique maximal ideal $\mathfrak m\subseteq R$.
We denote the residue map by $\overline{\phantom{x}}:R\to F$.

At some points in the text, we will assume $R$ to be \emph{2-Henselian}.
This means that if $x\in R^\ast$ is such that $\overline x$ is a square in $F$, then $x$ is a square in $R$.
This is fulfilled for example for complete discrete valuation rings, see \cite[Chapter II. (4.6) Hensel's Lemma]{MR1697859}.
In other situations, we will assume that the residue map $\overline{\phantom{x}}:R\to F$ has a \emph{section}, i.e. there is a (necessarily injective) ring homomorphism $t:F\to R\subseteq K$ such that for all $x\in R$, we have $\overline{t(\overline x)}=\overline x$.
In particular, by replacing $F$ with the isomorphic field $t(F)$, we can assume $F\subseteq K$ in this case. 

If $\varphi$ is a quadratic form over $K$, it is well known that $\varphi$ has a representation
\[\varphi\cong\varphi_1\perp t\varphi_2\]
where $\varphi_1,\varphi_2$ are \emph{unimodular}, i.e. they have a diagonalization such that each diagonal element is a unit in $R$. 

Applying the residue map $R\to F=R/\mathfrak m$ to the forms $\varphi_1,\varphi_2$, we obtain forms $\overline{\varphi_1}, \overline{\varphi_2}$ over $F$, called the \emph{first resp. second residue form} of $\varphi$.
Many properties of $\varphi$ can be read off its residue forms $\overline{\varphi_1},\overline{\varphi_2}$.

\begin{prop}[{\cite[Chapter 6, 2.6. Corollary]{Scharlau}, \cite[Lemma 19.5]{ElmanKarpenkoMerkurjev2008}}]\label{AnisotropCDV}
    Let $\varphi,\psi$ be two unimodular forms over $R$.
    \begin{enumerate}[label=(\alph*)]
        \item Let $\overline\varphi,\overline\psi$ be anisotropic over $F$. 
        Then $\varphi\perp t\psi$ is anisotropic over $K$.
        \item Let $R$ be 2-Henselian and $\varphi\perp t\psi$ be anisotropic over $K$. 
        Then $\overline\varphi$ and $\overline\psi$ are anisotropic.
    \end{enumerate} 
\end{prop}

The residue forms also behave well with respect to the powers of the fundamental ideal as the following result shows.

\begin{lem}[{\cite[Example 19.13, Lemma 19.14, Exercise 19.15]{ElmanKarpenkoMerkurjev2008}}]\label{lem:InValuationSequence}
    \begin{enumerate}[label=(\alph*)]
        \item\label{lem:InValuationSequence1} If $\varphi,\psi$ are unimodular forms with $[\varphi\perp t\psi]\in\I^n(K)$, we then have $[\varphi],[\psi]\in\I^{n-1}(K)$ and $[\overline\varphi],[\overline\psi]\in\I^{n-1}(F)$.
        In particular we have $[\varphi]\equiv[\psi]\mod\I^n(K)$ and $[\overline\varphi]\equiv[\overline\psi]\mod\I^n(F)$.
        \item\label{lem:InValuationSequence2} If $R$ is 2-Henselian, we have a split exact sequence
	    \begin{align*}
		  \begin{diagram}
			0 & \rTo & \I^n(F) & \rTo & \I^n(K) & \rTo & \I^{n-1}(F) & \rTo & 0
		\end{diagram}
	\end{align*}
	for all $n\in\N_+$, where the maps are given by lifting and taking the second residue class form.
    \end{enumerate}
\end{lem}

We now recall some facts relating the Pfister number of forms over $K$ and the Pfister numbers of its residue forms.
Note that these were stated originally only for complete discrete valuations, but the proofs transfer to the situation of a 2-Henselian valuation.

\begin{cor}[{\cite[Corollary 3.2]{MR4554566}}]\label{cor:LiftPfisterRepresentation}
	Let $R$ be 2-Henselian and $\varphi$ be a unimodular form over $R$ with $[\varphi]\in \I^n(F)$. 
    Then, the $n$-Pfister number of $\varphi$ over $F$ and of $\overline\varphi$ over $K$ coincide.
\end{cor}

\begin{prop}[{\cite[Proposition 3.3]{MR4554566}}]\label{prop:UnimodMalPfister}
	Let $R$ be 2-Henselian and $\psi$ be a unimodular form over $R$ with $[\psi]\in \I^{n-1}(K)$.
    For the form $\varphi=\Pfister t\otimes\psi$, we have ${\GP_{n-1}(\psi)=\GP_n(\varphi)}$.
\end{prop}

\section{Going down Theorems}

During this section, let $R$ be a discrete valuation ring with residue field $F$ and quotient field $K$ and let $n$ be an integer $\geq3$.
We assume that we can lift forms in $\I^k(F)$ to forms in $\I^k(K)$ for $k\in\{n-1, n\}$.
This is clearly fulfilled if the residue map $R\to F$ has a section.
It further known to be true if $R$ is 2-Henselian by \Cref{lem:InValuationSequence} \ref{lem:InValuationSequence2} in which case the lift is even unique.
Finally, we fix a uniformizer $t$.

\begin{prop}\label{thm:GAdown}
    Let $n\geq3$ and $d=2^n+2^{n-1}$. 
    If we have $\GA_n(K, d)=\I^n(K, d)$, we have $\GA_n(F, d)=\I^n(F, d)$ and $\GA_{n-1}(F, \frac d2)=\I^{n-1}(F, \frac d2)$.
\end{prop}
\begin{proof}
    We start with showing the equality $\GA_n(F, d)=\I^n(F,d)$ for $F$.
	Let $[\overline\varphi]\in \I^n(F)$ with $\overline\varphi$ an anisotropic form over $F$ of dimension $2^n+2^{n-1}$. 
    Let $\varphi$ be an $\I^n(K)$-lift of $\overline\varphi$.
    Thus there are Pfister forms $\pi_1,\pi_2\in \GP_n(K)$ such that we have 
    \[[\varphi]=[\pi_1]+[\pi_2]\in \W(K).\] 
    By \Cref{cor:LiftPfisterRepresentation}, $\overline\varphi$ satisfies {\Cref{prop:GAclassification}~\ref{prop:GAclassification1}}, which shows $\overline\varphi\in\GA_n(F)$ and thus the asserted equality.

    To show the equality $\GA_{n-1}(F, \frac d2)=\I^{n-1}(F, \frac d2)$, let now $[\overline\psi]\in \I^{n-1}(F)$ where $\overline\psi$ is an anisotropic form over $F$ of dimension $\frac d2=2^{n-1}+2^{n-2}$. 
    Let $\psi$ be an $\I^{n-1}(K)$-lift of $\psi.$
    Then $\Pfister t\otimes\psi$ is an anisotropic form of dimension $2^n+2^{n-1}$, see \Cref{AnisotropCDV}, and we clearly have 
    \[[\Pfister t\otimes\psi]\in\I^n(K).\] 
    This Witt class thus satisfies \Cref{prop:GAclassification} \ref{prop:GAclassification1}, i.e. $\Pfister t\otimes\psi$
    is Witt equivalent to the sum of two $n$-fold Pfister forms over $K$. 
    By \Cref{cor:LiftPfisterRepresentation} and \Cref{prop:UnimodMalPfister}, we know that already $\overline\psi$ satisfies \Cref{prop:GAclassification} \ref{prop:GAclassification1} which concludes the proof.
\end{proof}

We now prove a going down theorem for the property $\Sim$.

\begin{prop}\label{prop:SimGoDown}
    Let $R$ be 2-Henselian, $n\geq3$ be an integer and let $K$ fulfill $\Sim(n)$. 
    Then $F$ fulfills $\Sim(n-1)$ and $\Sim(n)$.
\end{prop}
\begin{proof}
    Let $\overline\varphi,\overline\psi$ be anisotropic forms over $F$ of dimension $2^n+2^{n-1}$ such that $[\overline\varphi],[\overline\psi]\in\I^n(F)$ and $[\overline\varphi]\equiv[\overline\psi]\mod\I^{n+1}(F)
    $ and let $\varphi,\psi$ denote their respective $\I^n(K)$-lifts.
    Since the difference $\overline\varphi\perp-\overline\psi$ can be lifted uniquely to an $\I^{n+1}(K)$ form, we have $[\varphi]\equiv[\psi]\mod\I^{n+1}(K)$.
    Thus, by assumption, $\varphi,\psi$ are similar over $K$ and since $\varphi,\psi$ are both unimodular, $\overline\varphi,\overline\psi$ are clearly similar over $F$.

    Let now $\overline\varphi,\overline\psi$ be anisotropic forms of dimension $2^{n-1}+2^{n-2}$ over $F$ such that $[\overline\varphi],[\overline\psi]\in\I^{n-1}(F)$ and $[\overline\varphi]\equiv[\overline\psi]\mod\I^{n}(F)$.
    Again, we denote by $\varphi,\psi$ their $\I^{n-1}(K)$-lifts.
    We consider $\alpha=\varphi\perp t\varphi$ and $\beta=\psi\perp t\psi$.
    We clearly have $[\alpha], [\beta]\in\I^n(K)$ and 
    \[[\alpha]-[\beta]=[\varphi\perp-\psi]+ t[\varphi\perp-\psi].\]
    As $[\overline\varphi\perp-\overline\psi]\in\I^n(F)$ we have $[\varphi\perp-\psi]\in \I^n(K)$ and it follows that $[\alpha]\equiv[\beta]\mod\I^{n+1}(K)$.
    By assumption on $K$, we conclude that $\alpha$ and $\beta$ are similar and by considering the residue class forms, we finally obtain that $\varphi$ and $\psi$ are similar.
\end{proof}

\section{Going up Theorems}

As above, let $R$ be a discrete valuation ring with residue field $F$, quotient field $K$ and uniformizer $t$.
We further assume $R$ to be 2-Henselian. 
Thus \Cref{lem:InValuationSequence} \ref{lem:InValuationSequence2} will play a crucial role in that we regularly use that $\I^n(F)$-forms correspond to unimodular $\I^n(K)$-forms via lifting respectively taking the first residue form.

\begin{thm}\label{thm:GAUp}
	Let $R$ be 2-Henselian.
    Let $n\geq3$ be an integer and $d=2^n+2^{n-1}$.
    If we have $\GA_n(F,d)=\I^n(F, d)$ and $\GA_{n-1}(F, \frac d2)=\I^{n-1}(F,\frac d2)$, then $\GA_n(K, d)=\I^n(K, d)$.
\end{thm}
\begin{proof}
	Let $\varphi$ be a quadratic form of dimension $2^n+2^{n-1}$ with $[\varphi]\in \I^n(K)$. 
    If $\varphi$ is similar to a unimodular form $\psi$, $\overline\psi$ is a form in $\I^n(F)$ of the same dimension according to \Cref{AnisotropCDV} and \Cref{lem:InValuationSequence}. 
    Since we have $\GA_n(F, d)=\I^n(F, d)$ and since we can lift a representation of $[\overline\psi]$ as a sum of the classes of two $n$-fold Pfister forms, this case is clear.
    
	Otherwise we can write $\varphi\cong\varphi_1\perp t\varphi_2$, with unimodular forms $\varphi_1,\varphi_2$ that fulfill $[\varphi_1],[\varphi_2]\in\I^{n-1}(K)$ due to \Cref{lem:InValuationSequence} \ref{lem:InValuationSequence1}. 
    After multipliying with $t$, we can assume the inequalities $0<\dim(\varphi_2)\leq\dim(\varphi_1)$. 
    As we have $[\overline{\varphi_1}],[\overline{\varphi_2}]\in \I^{n-1}(F)$ by \Cref{lem:InValuationSequence}, we have 
    \[\dim(\varphi_2)=\dim(\overline{\varphi_2})\in\{2^{n-1}, 2^{n-1}+2^{n-2}\}\] according to the Gap \Cref{thm:Gaps}. 
    If ${\dim(\varphi_2)=2^{n-1}}$, \Cref{prop:GAclassification} \ref{prop:GAclassification4} is fulfilled. 
    
	Otherwise we have $\dim(\varphi_1)=\dim(\varphi_2)=2^{n-1}+2^{n-2}$. 
    By \Cref{lem:InValuationSequence}, we have $[\overline{\varphi_1}]\equiv[\overline{\varphi_2}]\mod\I^n(F)$.
    Using the equality $\GA_{n-1}(F, \frac d2)=\I^{n-1}(F,\frac d2)$, \Cref{thm:GA->Sim} implies that $\overline{\varphi_1}$ and $\overline{\varphi_2}$ are similar.
    Therefore also $\varphi_1$ and $\varphi_2$ are similar and we obtain the existence of some $x\in K^\ast$ with $\varphi_2\cong -x\varphi_1$.
    This now implies
	\[\varphi \cong\varphi_1\perp t\varphi_2\cong \varphi_1\perp -tx\varphi_2\cong \Pfister{xt}\otimes\varphi_1\]
    and therefore, \Cref{prop:GAclassification} \ref{prop:GAclassification1} is fulfilled because of \Cref{prop:PfisterNumberMult}.
\end{proof}

\begin{rem}\label{rem:NoHigherDimension}
    It is not possible to establish results analogous to \Cref{thm:GAUp} for generalised Albert forms of higher dimension.
    We know from \cite[Example 5.3]{HoffmannTignol} that for $F=\C(X, Y)\laurent s$, every 14-dimensional form with Witt class in $\I^3(F)$ is a generalised Albert form, but over $K=F\laurent t$, there are 14-dimensional forms in $\I^3(F\laurent t)$ that are not generalised Albert forms.
    
    Expressed in our notation from above, we have $\GA_2(F, d)=\I^2(F,d)$ for all $d<8$ and $\GA_3(F, d)=\I^3(F,d)$ for all $d<16$, but $\GA_3(K, 14)\subsetneq\I^3(K, 14)$.
\end{rem}

\begin{cor}\label{cor:LinkedLaurentGA} 
	Let $E$ be a linked field, $I$ be a totally ordered index set and $F=E\laurent{t_i}_{i\in I}$ be an iterated Laurent extension of $E$. 
    Let $n\geq2$ be an integer and $d=2^n+2^{n-1}$.
    We then have $\GA_n(F,d)=\I^n(F,d)$ and $F$ fulfills $\Sim(n)$.
\end{cor}
\begin{proof}
	For $n\in\{2,3\}$, there is nothing more to show.
	With \Cref{thm:GAUp} in mind, it is enough to verify the equality $\GA_n(F, d)=\I^n(F, d)$ for all $n\in\N$ with $n\geq 3$ for linked fields. 
    But this is trivially fulfilled as over these fields, there are no anisotropic forms of dimension $2^n+2^{n-1}$ with Witt class in $\I^n(F)$, see \Cref{cor:nlinkedPfisterNumber} 
\end{proof}

The same proof as in \Cref{cor:LinkedLaurentGA} clearly also yields that iterated Laurent extensions $F$ of $\mathbb F_3, \R, \C$ fulfil $\GA_n(F, d)=\I^n(F, d)$ and have property $\Sim(n)$ for all $n\geq2$ and $d=2^n+2^{n-1}$.

The occuring Witt rings are then group ring extensions of $\Z/4\Z, \Z, \Z/2\Z$ respectively with a group of exponent 2. 
Fields with Witt ring isomorphic to such a group ring extension are characterized by the fact that each anisotropic form of dimension 2 represents at most 2 square classes, see \cite[Theorem 1.9]{Ware}.
Such fields are thus called \emph{rigid}.
The author studied rigid fields in \cite{MR4554566} with respect to their Pfister numbers.

By \cite[Theorem 2.3]{Cordes1973} fields with isomorphic Witt rings are even \emph{equivalent with respect to quadratic forms} (or for short \emph{$q$-equivalent}, see \cite[Chapter XII. Definition 1.1]{Lam2005}) in the sense of the just cited article.
Such a $q$-equivalence between two fields $F, K$ yields a correspondence between isometry types of quadratic forms over $F$ resp. $K$ that preserves orthogonal sums, tensor products, dimensions, anisotropicity and Witt indices and thus also an isomorphism between $\W(F)$ and $\W(K)$ that maps classes of 1-dimensional forms to classes of 1-dimensional forms.

From this, it can be seen that, given that two fields $F_1$ and $F_2$ are $q$-equivalent, for all $n, d$ we have $\GA_n(F_1,d)=\I^n(F_1, d)$ if and only if $\GA_n(F_2, d)=\I^n(F_2, d)$ and further, $F_1$ fulfills $\Sim(n)$ if and only if $F_2$ does.

Since rigid fields are $q$-equivalent to an iterated Laurent extension of $\mathbb F_3, \R$ or $\C$ as mentioned above, we particularly have the following result.

\begin{cor}\label{cor:RigidGA}
    Let $F$ be a rigid field, $n\geq2$ be an integer and $d=2^n+2^{n-1}$. 
    We then have $\GA_n(F, d)=\I^n(F, d)$ and $F$ fulfills $\Sim(n)$ for all $n\geq2$.
\end{cor}

We now turn to the property $\Sim(n)$.

\begin{thm}\label{thm:SimGoingUp}
	Let $R$ be 2-Henselian and let $n\geq 3$ be an integer. If $F$ fulfils $\Sim(n-1)$ and $\Sim(n)$ then $K$ fulfils $\Sim(n)$.
\end{thm}
\begin{proof}
	Let $\varphi_1,\varphi_2$ be anisotropic forms of dimension $2^n+2^{n-1}$ with $[\varphi_1],[\varphi_2]\in\I^n(K)$ and $[\varphi_1]\equiv[\varphi_2]\mod\I^{n+1}(K)$.
    For $i\in\{1,2\}$, we can write $\varphi_i\cong\alpha_i\perp t\beta_i$ with unimodular forms $\alpha_i,\beta_i$.
    Possibly after scaling the $\varphi_i$ with $t$, we may further assume
    \[{2\cdot\dim(\beta_i)\leq\dim(\varphi_i)=2^n+2^{n-1}}.\]  
    By \Cref{lem:InValuationSequence} \ref{lem:InValuationSequence1}, we have $[\alpha_i],[\beta_i]\in\I^{n-1}(K)$ for $i\in\{1,2\}$.
	We further have
	\begin{align*}
		[(\alpha_1\perp-\alpha_2)\perp t(\beta_1\perp-\beta_2)]=[\varphi_1\perp-\varphi_2]\in \I^{n+1}(K),
	\end{align*}
	which leads to
	\begin{align*}
        [\alpha_1] \equiv [\alpha_2] \equiv [\beta_1] \equiv [\beta_2] \mod \I^{n}(K)
    \end{align*}
    and thus also to
    \begin{align}\label{eq:Kongruences}
        [\overline{\alpha_1}]\equiv[\overline{\alpha_2}] \equiv[\overline{\beta_1}] \equiv[\overline{\beta_2}]\mod \I^{n}(F).
    \end{align}

	As we assume $F$ to fulfil $\Sim(n-1)$, there is some $\overline y\in F^\ast$ with $\overline{\beta_2}\cong \overline y\overline{\beta_1}$. 
    In particular, this implies $\overline{\beta_1}=0$ if and only if $\overline{\beta_2}=0$, i.e. $\beta_1=0$ if and only if $\beta_2=0$. 
    So if $\beta_1=0$, both $\varphi_1$ and $\varphi_2$ are unimodular and the assertion now follows from $\Sim(n)$ for $F$ using that similarity is preserved under lifting.

	By the Gap \Cref{thm:Gaps}, the remaining cases are given by  \[\dim(\beta_1)=\dim(\beta_2)=2^{n-1}\text{ and }\dim(\beta_1)=\dim(\beta_2)=2^{n-1}+2^{n-2}.\] 
    We will now show the existence of some $x\in K^\ast$ with $\alpha_2\cong x\alpha_1$ for both cases separately.
    Note that since we consider a 2-Henselian valuation, it is enough to show the existence of some $\overline x\in F^\ast$ with $\overline{\alpha_2}\cong\overline x\overline{\alpha_1}$ since lifting then yields the assertion.
    
	In the case $\dim(\beta_1)=2^{n-1}+2^{n-2}$, we have $\dim(\overline{\alpha_1})=\dim(\overline{\alpha_2})=2^{n-1}+2^{n-2}$ and $[\overline{\alpha_1}] \equiv [\overline{\alpha_2}]\mod \I^n(F)$ by \eqref{eq:Kongruences}. 
    There is thus some $\overline x\in F^\ast$ with $\overline{\alpha_2}\cong \overline{x}\overline{\alpha_1}$ by $\Sim(n-1)$ for $F$.
    
	In the case where $\dim(\beta_1)=2^{n-1}$, we have ${\dim(\overline{\alpha_1}) = \dim(\overline{\alpha_2})=2^n}$.
    Since we have $[\overline{\alpha_i}] \equiv [\overline{\beta_i}]\mod\I^n(F)$ with $\overline{\beta_i} \in \GP_{n-1}(F)$, we can deduce by \cite[Proposition 3.11 (i)]{HoffmannTwistedPfister} that $\overline{\alpha_1}$ and $\overline{y}\overline{\alpha_2}$ are twisted Pfister forms.
    We further have
    \[[\overline{\alpha_1}]-[\overline{y}\overline{\alpha_2}]=[\overline{\varphi_1}]-[\overline{y}\overline{\varphi_2}] \in\I^{n+1}(F).\]
     
    Thus the existence of $\overline x\in F^\ast$ with $\overline{\alpha_2}\cong \overline{x}\overline{\alpha_1}$ follows from \Cref{prop:TwistedPfisterHalfNeighborSimilar}.

    As remarked before, we now obtain the existence of some $x\in K^\ast$ with $\alpha_2\cong x\alpha_1$ by lifting.
	In $\W(K)$ we then have
	\begin{align*}
		t([\beta_1]-[x\beta_2])&=[\varphi_1]-[x\varphi_2]\in\I^{n+1}(K)
	\end{align*}
	We thus have $[\beta_1\perp-x\beta_2]\in \I^{n+1}(K)$. 
    But as the anisotropic part of $\beta_1\perp-x\beta_2$ has dimension at most ${3\cdot2^{n-1}<2^{n+1}}$, it is hyperbolic by the Arason-Pfister Hauptsatz \ref{thm:APH}. 
    We thus have $\beta_1\cong x\beta_2$, which is equivalent to $\beta_2\cong x\beta_1$. 
    Summarizing, we have
	$$x\varphi_1\cong x\alpha_1\perp tx\beta_1\cong \alpha_2\perp t\beta_2\cong \varphi_2,$$
	as claimed.
\end{proof}

To summarise, we combine our main results \Cref{thm:GAdown}, \Cref{prop:SimGoDown}, \Cref{thm:GAUp}, \Cref{thm:SimGoingUp} to obtain the following handy equivalences:

\begin{cor}
    Let $R$ be a 2-Henselian valuation ring with residue field $F$ and quotient field $K$ both of characteristic not 2.
    Let further $n\geq2$ and $d=2^n+2^{n-1}$.
    \begin{enumerate}[label=(\alph*)]
        \item We have $\GA_n(K, d)=\I^n(K, d)$ if and only if $\GA_n(F, d)=\I^n(F, d)$ and $\GA_{n-1}(F, \frac d2)=\I^{n-1}(F, \frac d2)$.
        \item $K$ has propoerty $\Sim(n)$ if and only if $F$ has properties $\Sim(n)$ and $\Sim(n-1)$.
    \end{enumerate}
\end{cor}

\subsubsection*{Declaration of interests}

The authors declare that they have no known competing financial interests or personal relationships that could have appeared to influence the work reported in this paper.

\bibliographystyle{alpha}
\bibliography{literatur}

\end{document}